\documentclass{amsart}%
\usepackage{amssymb}
\usepackage{amsfonts}
\usepackage{amsmath}
\usepackage{graphicx}%
\newtheorem{theorem}{Theorem}
\theoremstyle{plain}

\newtheorem{corollary}{Corollary}

\newtheorem{lemma}{Lemma}

\newtheorem{proposition}{Proposition}
\newtheorem{remark}{Remark}

\numberwithin{equation}{section}
\begin{document}
\title[ ]{A Simpler Proof of Frank and\ Lieb's Sharp Inequality on the Heisenberg Group}
\author{Fengbo Hang}
\address{Courant Institute, 251 Mercer Street, New York, NY 10012}
\email{fengbo@cims.nyu.edu}
\author{Xiaodong Wang}
\address{Department of Mathematics, Michigan State University, East Lansing, MI 48864}
\email{xwang@math.msu.edu}

\begin{abstract}
We give a simpler proof of the sharp Frank-Lieb inequality on the Heisenberg
group $\mathbb{H}^{m}$. The proof bypasses the sophisticated argument for
existence of a minimizer and is based on the study of the 2nd variation of
subcritical functionals using their fundamental techniques.

\end{abstract}
\maketitle

\section{Introduction}

In a ground-breaking work \cite{FL1}, Frank and Lieb determined the sharp
constant and extremal functions for the Folland-Stein inequality on the
Heisenberg group $\mathbb{H}^{m}$. Recall
\[
\mathbb{H}^{m}=\left\{  u=\left(  z,t\right)  :z\in\mathbb{C}^{m}%
,t\in\mathbb{R}\right\}
\]
with the group law
\[
u\cdot u^{\prime}=\left(  z,t\right)  \cdot\left(  z^{\prime},t\right)
=\left(  z+z^{\prime},t+t^{\prime}+2\operatorname{Im}z\overline{z^{\prime}%
}\right)  .
\]
The Haar measure on $\mathbb{H}^{m}$ is the standard Lebesgue measure
$du=dzdt$. For $\delta>0$ we write $\delta u=\left(  \delta z,\delta
^{2}t\right)  $ for the dilation. We denote the homogeneous norm on
$\mathbb{H}^{m}$ by
\[
\left\vert u\right\vert =\left\vert \left(  z,t\right)  \right\vert =\left(
\left\vert z\right\vert ^{4}+t^{2}\right)  ^{1/4}.
\]
Throughout the paper $Q=2m+2$. The Frank-Lieb inequality then states

\begin{theorem}
Let $0<\lambda<Q$ and $p=2Q/\left(  2Q-\lambda\right)  $. Then for any $f,g\in
L^{p}\left(  \mathbb{H}^{m}\right)  $%
\[
\left\vert \int_{\mathbb{H}^{m}\times\mathbb{H}^{m}}\frac{\overline
{f(u)}g\left(  v\right)  }{\left\vert u^{-1}v\right\vert ^{\lambda}%
}dudv\right\vert \leq\left(  \frac{2\pi^{m+1}}{m!}\right)  ^{\lambda/Q}%
\frac{m!\Gamma\left(  \left(  Q-\lambda\right)  /2\right)  }{\Gamma^{2}\left(
\left(  2Q-\lambda\right)  /4\right)  }\left\Vert f\right\Vert _{p}\left\Vert
g\right\Vert _{p},
\]
with equality if and only if%
\[
f\left(  u\right)  =cH\left(  \delta\left(  a^{-1}u\right)  \right)  ,g\left(
u\right)  =c^{\prime}H\left(  \delta\left(  a^{-1}u\right)  \right)
\]
for some $c,c^{\prime}\in\mathbb{C}$ and $a\in\mathbb{H}^{m}$ (unless
$f\equiv0$ or $g\equiv0$). Here $H$ is the function defined by%
\[
H\left(  z,t\right)  =\left(  \left(  1+\left\vert z\right\vert ^{2}\right)
^{2}+t^{2}\right)  ^{-\left(  2Q-\lambda\right)  /4}.
\]

\end{theorem}

Via the Cayley transform, Theorem 1 is equivalent to the following formulation
on the sphere $\mathbb{S}^{2m+1}=\left\{  z\in\mathbb{C}^{m+1}:\left\vert
z\right\vert =1\right\}  $.

\begin{theorem}
Let $0<\lambda<Q$ and $p=2Q/\left(  2Q-\lambda\right)  $. Then for any $f,g\in
L^{p}\left(  \mathbb{S}^{2m+1}\right)  $%
\[
\left\vert \int_{\mathbb{S}^{2m+1}\times\mathbb{S}^{2m+1}}\frac{\overline
{f(\xi)}g\left(  \eta\right)  }{\left\vert 1-\xi\cdot\overline{\eta
}\right\vert ^{\lambda/2}}d\sigma\left(  \xi\right)  d\sigma\left(
\eta\right)  \right\vert \leq\left(  \frac{2\pi^{m+1}}{m!}\right)
^{\lambda/Q}\frac{m!\Gamma\left(  \left(  Q-\lambda\right)  /2\right)
}{\Gamma^{2}\left(  \left(  2Q-\lambda\right)  /4\right)  }\left\Vert
f\right\Vert _{p}\left\Vert g\right\Vert _{p},
\]
with equality if and only if%
\[
f\left(  \xi\right)  =\frac{c}{\left\vert 1-\xi\cdot\overline{\eta}\right\vert
^{\left(  2Q-\lambda\right)  /2}},g\left(  \xi\right)  =\frac{c^{\prime}%
}{\left\vert 1-\xi\cdot\overline{\eta}\right\vert ^{\left(  2Q-\lambda\right)
/2}}%
\]
for some $c,c^{\prime}\in\mathbb{C}$ and $\eta\in\mathbb{C}^{m+1}$ with
$\left\vert \eta\right\vert <1$ (unless $f\equiv0$ or $g\equiv0$).
\end{theorem}

Their proof consists of two major steps. The first step is to prove that the
infimum
\[
\inf\left\{  \left\vert \int_{\mathbb{H}^{m}\times\mathbb{H}^{m}}%
\frac{\overline{f(u)}g\left(  v\right)  }{\left\vert u^{-1}v\right\vert
^{\lambda}}dudv\right\vert :\left\Vert f\right\Vert _{p}=\left\Vert
g\right\Vert _{p}=1\right\}
\]
is achieved by some $\left(  f,g\right)  $ and moreover $f=g$. This requires
some sophisticated harmonic analysis on $\mathbb{H}^{m}$. The 2nd step is to
work on $\mathbb{S}^{2m+1}$ and determine the extremal function $f=g$. By
using the invariance of the problem under the CR automorphism group and a
Hersch-type argument, they first arrange that $f$ satisfies a moment zero
condition. Then by exploiting masterfully the 2nd variation of the functional
with test functions provided by the moment zero condition, they prove that
such $f$ must be constant.

In this paper, we present a shorter proof of the Frank-Lieb inequality which
bypasses the subtle proof of existence and the Hersch-type argument. We use a
scheme of subcritical approximation. The starting point is that the operator%
\[
I_{\lambda}f\left(  \xi\right)  =\int_{\mathbb{S}^{2m+1}}\frac{f\left(
\eta\right)  }{\left\vert 1-\xi\cdot\overline{\eta}\right\vert ^{\lambda/2}%
}d\sigma\left(  \eta\right)  .
\]
is compact from $L^{p}\left(  \mathbb{S}^{2m+1}\right)  $ to $L^{p^{\prime}%
}\left(  \mathbb{S}^{2m+1}\right)  $, if $p>2Q/\left(  2Q-\lambda\right)  $
and $p^{\prime}=p/\left(  p-1\right)  $. Therefore the minimization problem
\[
\Lambda_{p}=\inf\left\{  \left\Vert I_{\lambda}f\right\Vert _{p^{\prime}%
}:\left\Vert f\right\Vert _{p}=1\right\}
\]
has a minimizer $u_{p}$ which can be taken to be nonnegative. Moreover, due to
a symmetry-breaking, $u_{p}$ automatically satisfies a moment zero condition.
Therefore we can analyze the 2nd variation of the functional $\left\Vert
I_{\lambda}f\right\Vert _{p^{\prime}}/\left\Vert f\right\Vert _{p}$ at $u_{p}$
by fully using Frank and Lieb's techniques. Though we could not prove that
$u_{p}$ is constant as we had expected, we are able to show that $u_{p}$
converges to a constant in $L^{2Q/\left(  2Q-\lambda\right)  }\left(
\mathbb{S}^{2m+1}\right)  $ as $p\searrow2Q/\left(  2Q-\lambda\right)  $. This
is enough to yield Frank-Lieb's sharp inequality.

The paper is organized as follows. In Section 2, we present a proof for
Jerison-Lee's sharp Sobolev inequality which is equivalent to Frank-Lieb's
inequality with $\lambda=Q-2$. The analysis is simpler in this special case.
In Section 3, we collect some fundamental results on the operator $I_{\lambda
}$. We present the proof for Frank-Lieb's inequality in Section 4. Finally, we
make some further remarks and raise several open problems in the last Section.

\textbf{Acknowledgment. }We would like to thank Rupert Frank for useful suggestions.

\section{Proof of Jerison-Lee inequality}

The Frank-Lieb inequality for $\lambda=Q-2$ is equivalent to the following
sharp Sobolev inequality established by Jerison and Lee in 1990's%

\begin{equation}
\int_{\mathbb{S}^{2m+1}}\left(  \left\vert \nabla_{b}f\right\vert ^{2}%
+\frac{m^{2}}{4}f^{2}\right)  d\sigma\geq\frac{\pi m^{2}}{4}\left(  \frac
{2}{m!}\right)  ^{\frac{1}{m+1}}\left(  \int_{\mathbb{S}^{2m+1}}\left\vert
f\right\vert ^{2\left(  m+1\right)  /m}d\sigma\right)  ^{\frac{m}{m+1}}.
\label{JLE}%
\end{equation}
Here we use the canonical pseudohermitian structure $\theta_{c}=\left(
2\sqrt{-1}\overline{\partial}\left\vert z\right\vert ^{2}\right)
|_{\mathbb{S}^{2m+1}}$ on $\mathbb{S}^{2m+1}$ with Webster scalar curvature
$R=m\left(  m+1\right)  /2$ and the adapted Riemannian metric $4g_{0}$, where
$g_{0}$ is the standard metric on $\mathbb{S}^{2m+1}$. (But to be consistent
with the general case, we still use the standard measure $d\sigma$ which
differs from the usual pseudohermitian volume form $\theta_{c}\wedge\left(
d\theta_{c}\right)  ^{m}$ by a scaling constant.) To prove this fundamental
result, Jerison and Lee \cite{JL1} first proved that the sharp constant%
\begin{equation}
\Lambda=\inf\left\{  E\left(  f\right)  :\left\Vert f\right\Vert _{2\left(
m+1\right)  /m}=1\right\}  , \label{CRY}%
\end{equation}
where $E\left(  f\right)  =\int_{\mathbb{S}^{2m+1}}\left(  \left\vert
\nabla_{b}f\right\vert ^{2}+\frac{m^{2}}{4}f^{2}\right)  d\sigma$, is achieved
by a smooth and positive minimizer $u$ satisfying the following PDE%
\begin{equation}
-\Delta_{b}u+\frac{m^{2}}{4}u=\Lambda u^{\left(  m+2\right)  /m}\text{ on
}\mathbb{S}^{2m+1}. \label{pde}%
\end{equation}
This was achieved by a blow-up analysis which we outline. For any $q<\left(
m+2\right)  /m$, the embedding from $S^{1,2}\left(  \mathbb{S}^{2m+1}\right)
$ to $L^{q+1}\left(  \mathbb{S}^{2m+1}\right)  $ is compact and therefore the
minimization problem
\[
\Lambda_{q}=\inf\left\{  E\left(  f\right)  :\left\Vert f\right\Vert
_{q+1}=1\right\}
\]
has a smooth and positive solution $u_{q}$ which then satisfies the PDE%
\begin{equation}
-\Delta_{b}u+\frac{m^{2}}{4}u=\Lambda_{q}u^{q}\text{ on }\mathbb{S}^{2m+1}.
\label{pdeq}%
\end{equation}
and $\left\Vert u_{q}\right\Vert _{q+1}=1$. If $\left\{  u_{q}\right\}  $ does
not blow up as $q\nearrow\left(  m+2\right)  /m$, the limit is a solution for
(\ref{CRY}). If it does blow up, then one can properly scale and extract a
limit on $\mathbb{H}^{m}$. As $\mathbb{H}^{m}$ is equivalent to $\mathbb{S}%
^{2m+1}$, this limit also yields a solution to (\ref{CRY}). In \cite{JL2} they
proved that any positive solution of the PDE (\ref{pde}) must be of the form%
\[
u\left(  \xi\right)  =c\left\vert \cosh t+\left(  \sinh t\right)  \xi
\cdot\overline{\eta}\right\vert ^{-2}%
\]
for some $t\geq0$ and $\eta\in\mathbb{S}^{2m+1}$. The inequality (\ref{JLE})
then follows.

In Section 3 of \cite{FL1}, assuming existence of a minimizer for (\ref{CRY})
Frank and Lieb presented a simpler proof for the classification of all
extremal functions. We will apply their idea to show directly that $\left\{
u_{q}\right\}  $ converges to a constant as $q\nearrow\left(  m+2\right)  /m$.
Therefore constant functions are minimizers for (\ref{CRY}) and hence the
sharp inequality(\ref{JLE}).

For $\eta\in\mathbb{S}^{2m+1}$ and $t\geq0$ we define $\Phi_{t,\eta
}:\mathbb{S}^{2m+1}\rightarrow\mathbb{S}^{2m+1}$ by%
\[
\Phi_{t,\eta}\left(  \xi\right)  =\frac{1}{\cosh t+\sinh t\xi\cdot
\overline{\eta}}\left(  \xi-\left(  \xi\cdot\overline{\eta}\right)
\eta\right)  +\frac{\sinh t+\cosh t\xi\cdot\overline{\eta}}{\cosh t+\sinh
t\xi\cdot\overline{\eta}}\eta.
\]
This is a one-parameter family of CR automorphisms with $\Phi_{0,\eta}=Id$.
Moreover, $\Phi_{t,\eta}^{\ast}\theta_{c}=\phi_{t,\eta}\theta_{c}$, where%
\[
\phi_{t,\eta}\left(  \xi\right)  =\frac{1}{\left\vert \cosh t+\left(  \xi
\cdot\overline{\eta}\right)  \sinh t\right\vert ^{2}}.
\]
(For more details we refer to the appendices of \cite{FL1}.) Given a function
$f$, we get a family $f_{t,\eta}$ defined by%
\[
f_{t,\eta}\left(  \xi\right)  =f\circ\Phi_{t,\eta}\left(  \xi\right)
\phi_{t,\eta}^{m/2}\left(  \xi\right)  .
\]
It is well known that the CR Yamabe functional is invariant under such
transformations, more specifically
\[
E\left(  f_{t,\eta}\right)  =E\left(  f\right)  ,\left\Vert f_{t,\eta
}\right\Vert _{2\left(  m+1\right)  /m}=\left\Vert f\right\Vert _{2\left(
m+1\right)  /m}.
\]
With this invariance, we can establish a CR analogue of the Kazdan-Warner
identity \cite{KW}.

\begin{proposition}
\label{kwi} If a positive function $f\in C^{\infty}\left(  \mathbb{S}%
^{2m+1}\right)  $ satisfies the following PDE%
\begin{equation}
-\Delta_{b}f+\frac{m^{2}}{4}f=Kf^{\left(  m+2\right)  /m}, \label{kpde}%
\end{equation}
where $K$ is a smooth function on $\mathbb{S}^{2m+1}$, then
\begin{equation}
\int_{\mathbb{S}^{2m+1}}\left[  \left\langle \nabla K,\nabla\psi_{\eta
}\right\rangle _{0}+TKT\psi_{\eta}\right]  f^{2\left(  m+1\right)  /m}%
d\sigma=0, \label{KW}%
\end{equation}
where
\[
\psi_{\eta}\left(  \xi\right)  =\operatorname{Re}\left(  \xi\cdot
\overline{\eta}\right)  ,T=J\xi
\]
and $\left\langle \cdot,\cdot\right\rangle _{0}$ stands for the standard
metric on $\mathbb{S}^{2m+1}$.
\end{proposition}

\begin{proof}
As $f$ satisfies the equation (\ref{kpde}), it is a critical point of the
functional
\[
\mathcal{F}\left(  u\right)  =E\left(  u\right)  /\left(  \int_{\mathbb{S}%
^{2m+1}}Ku^{2\left(  m+1\right)  /m}d\sigma\right)  ^{m/\left(  m+1\right)
}.
\]
Thus%
\[
\frac{d}{dt}|_{t=0}\mathcal{F}\left(  f_{t,\eta}\right)  =0.
\]
But $E\left(  f_{t,\eta}\right)  =E\left(  f\right)  $ and, by a change of
variables,
\begin{align*}
\int_{\mathbb{S}^{2m+1}}Kf_{t,\eta}^{2\left(  m+1\right)  /m}d\sigma &
=\int_{\mathbb{S}^{2m+1}}K\left(  f\circ\Phi_{t,\eta}\right)  ^{2\left(
m+1\right)  /m}\phi_{t,\eta}^{m+1}d\sigma\\
&  =\int_{\mathbb{S}^{2m+1}}K\circ\Phi_{t,\eta}^{-1}f^{2\left(  m+1\right)
/m}d\sigma.
\end{align*}
Therefore we obtain%
\[
\int_{\mathbb{S}^{2m+1}}\left\langle \nabla K,X\right\rangle _{0}f^{2\left(
m+1\right)  /m}d\sigma=0,
\]
with $X=\frac{d}{dt}|_{t=0}\Phi_{t,\eta}$. Direct calculation yields%
\[
X=\nabla\psi_{\eta}+\left(  T\psi_{\eta}\right)  T
\]
and hence the formula (\ref{KW})
\end{proof}

\begin{corollary}
\label{mzero}If $u>0$ satisfies (\ref{pdeq}) with $1<q<\left(  m+2\right)
/m$, then
\[
\int_{\mathbb{S}^{2m+1}}u\left(  \xi\right)  ^{q+1}\xi d\sigma\left(
\xi\right)  =0.
\]

\end{corollary}

\begin{proof}
The equation (\ref{pdeq}) can be written as%
\[
-\Delta_{b}u+\frac{m^{2}}{4}u=Ku^{\left(  m+2\right)  /m},
\]
with $K=\Lambda_{q}u^{q-\left(  m+2\right)  /m}$. By Proposition \ref{kwi} and
integration by parts%
\begin{align*}
0  &  =\int_{\mathbb{S}^{2m+1}}\left[  \left\langle \nabla K,\nabla\psi_{\eta
}\right\rangle _{0}+TKT\psi_{\eta}\right]  u^{2\left(  m+1\right)  /m}%
d\sigma\\
&  =\left(  q-\frac{m+2}{m}\right)  \Lambda_{q}\int_{\mathbb{S}^{2m+1}}\left[
\left\langle \nabla u,\nabla\psi_{\eta}\right\rangle _{0}+TuT\psi_{\eta
}\right]  u^{q}d\sigma\\
&  =\frac{1}{q+1}\left(  q-\frac{m+2}{m}\right)  \Lambda_{q}\int
_{\mathbb{S}^{2m+1}}\left[  \left\langle \nabla u^{q+1},\nabla\psi_{\eta
}\right\rangle _{0}+Tu^{q+1}T\psi_{\eta}\right]  d\sigma\\
&  =-\frac{1}{q+1}\left(  q-\frac{m+2}{m}\right)  \Lambda_{q}\int
_{\mathbb{S}^{2m+1}}u^{q+1}\left(  \Delta\psi_{\eta}+T^{2}\psi_{\eta}\right)
d\sigma\\
&  =\frac{2\left(  m+1\right)  }{q+1}\left(  q-\frac{m+2}{m}\right)
\Lambda_{q}\int_{\mathbb{S}^{2m+1}}u^{q+1}\psi_{\eta}d\sigma.
\end{align*}
Therefore $\int_{\mathbb{S}^{2m+1}}u^{q+1}\psi_{\eta}d\sigma=0$ for all
$\eta\in\mathbb{S}^{2m+1}$. This yields the desired conclusion.
\end{proof}

By calculating the 2nd variation of the functional at the minimizer $u_{q}$,
we have for any $f$ with $\int_{\mathbb{S}^{2m+1}}u_{q}^{q}fd\sigma=0$%
\[
\int_{\mathbb{S}^{2m+1}}\left(  \left\vert \nabla_{b}f\right\vert ^{2}%
+\frac{m^{2}}{4}f^{2}\right)  d\sigma\geq q\int_{\mathbb{S}^{2m+1}}\left(
\left\vert \nabla_{b}u_{q}\right\vert ^{2}+\frac{m^{2}}{4}u_{q}^{2}\right)
d\sigma\int_{\mathbb{S}^{2m+1}}u_{q}^{q-1}f^{2}d\sigma.
\]
By Corollary \ref{mzero}, we can take $f\left(  \xi\right)  =u_{q}\left(
\xi\right)  x_{i}$ or $u_{q}\left(  \xi\right)  y_{i}$, where $x_{i}%
=\operatorname{Re}\xi_{i},y_{i}=\operatorname{Im}\xi_{i}$. Therefore, summing
the corresponding inequalities for all such $f$'s, we obtain, in view of
$\left\Vert u_{q}\right\Vert _{q+1}=1$
\begin{align*}
q\int_{\mathbb{S}^{2m+1}}\left(  \left\vert \nabla_{b}u_{q}\right\vert
^{2}+\frac{m^{2}}{4}u_{q}^{2}\right)  d\sigma &  \leq\int_{\mathbb{S}^{2m+1}%
}\left[  \sum_{i}\left\vert \nabla_{b}\left(  u_{q}x_{i}\right)  \right\vert
^{2}+\left\vert \nabla_{b}\left(  u_{q}y_{i}\right)  \right\vert ^{2}%
+\frac{m^{2}}{4}u_{q}^{2}\right]  d\sigma\\
&  =\int_{\mathbb{S}^{2m+1}}\left[  \left\vert \nabla_{b}u_{q}\right\vert
^{2}+u_{q}^{2}\left(  \frac{m^{2}}{4}-\sum_{i}\left(  x_{i}\Delta_{b}%
x_{i}+y_{i}\Delta_{b}y_{i}\right)  \right)  \right]  d\sigma\\
&  =\int_{\mathbb{S}^{2m+1}}\left[  \left\vert \nabla_{b}u_{q}\right\vert
^{2}+\frac{m\left(  m+2\right)  }{4}u_{q}^{2}\right]  d\sigma.
\end{align*}
Therefore
\[
\left(  q-1\right)  \int_{\mathbb{S}^{2m+1}}\left\vert \nabla_{b}%
u_{q}\right\vert ^{2}d\sigma\leq\frac{m^{2}}{4}\left(  \frac{m+2}{m}-q\right)
\int_{\mathbb{S}^{2m+1}}u_{q}^{2}d\sigma.
\]
As $\left\Vert u_{q}\right\Vert _{q+1}=1$, the above inequality implies that
$\int_{\mathbb{S}^{2m+1}}\left\vert \nabla_{b}u_{q}\right\vert ^{2}%
d\sigma\rightarrow0$ as $q\nearrow\frac{m+2}{m}$. It follows that we can
choose a sequence $q_{i}\nearrow\frac{m+2}{m}$ s.t. the sequence $\left\{
u_{i}=u_{q_{i}}\right\}  $ converges to a nonzero constant $c$, i.e.
\[
\int_{\mathbb{S}^{2m+1}}\left\vert \nabla_{b}\left(  u_{i}-c\right)
\right\vert ^{2}d\sigma\rightarrow0,\left\Vert u_{i}-c\right\Vert _{2\left(
m+1\right)  /m}\rightarrow0.
\]
Therefore the constant function $c$ is a minimizer of (\ref{CRY}) and the
inequality (\ref{JLE}) follows.

\section{Preliminaries for the general case}

In this Section, we present some fundamental technical results needed for the
proof of the Frank-Lieb inequality in the general case. We first recall the
Funk-Hecke theorem on $\mathbb{S}^{2m+1}$ (cf. \cite{FL1, BLM} and references
therein for more details). The space $L^{2}\left(  \mathbb{S}^{2m+1}\right)  $
can be decomposed into its $U\left(  m+1\right)  $-irreducible components%
\begin{equation}
L^{2}\left(  \mathbb{S}^{2m+1}\right)  =\oplus_{j,k\geq0}\mathcal{H}_{j,k},
\label{l2d}%
\end{equation}
Here $\mathcal{H}_{j,k}$ is the space of restrictions to $\mathbb{S}^{2m+1}$
of harmonic polynomials $p\left(  z,\overline{z}\right)  $ on $\mathbb{C}%
^{m+1}$ that are homogeneous of degree $j$ in $z$ and degree $k$ in
$\overline{z}$. For an integrable function $K$ on the unit disc in
$\mathbb{C}$ we can define an integral operator with kernel $K\left(  \xi
\cdot\overline{\eta}\right)  $ on $\mathbb{S}^{2m+1}$ by%
\[
\left(  If\right)  \left(  \xi\right)  =\int_{\mathbb{S}^{2m+1}}K\left(
\xi\cdot\overline{\eta}\right)  f\left(  \eta\right)  d\sigma\left(
\eta\right)  .
\]
The Funk-Hecke theorem states that such operators are diagonal with respect to
the decomposition (\ref{l2d}). We need the following explicit results.

\begin{proposition}
\label{weights}(Corollary 5.3 in \cite{FL1}) Let $-1<\alpha<\left(
m+1\right)  /2$.

\begin{enumerate}
\item The eigenvalue of the operator with kernel $\left\vert 1-\xi
\cdot\overline{\eta}\right\vert ^{-2\alpha}$ on the subspace $\mathcal{H}%
_{j,k}$ is
\[
E_{j,k}=\frac{2\pi^{m+1}\Gamma\left(  m+1-2\alpha\right)  }{\Gamma^{2}\left(
\alpha\right)  }\frac{\Gamma\left(  j+\alpha\right)  }{\Gamma\left(
j+m+1-\alpha\right)  }\frac{\Gamma\left(  k+\alpha\right)  }{\Gamma\left(
k+m+1-\alpha\right)  }%
\]

\item The eigenvalue of the operator with kernel $\left\vert \xi\cdot
\overline{\eta}\right\vert ^{2}\left\vert 1-\xi\cdot\overline{\eta}\right\vert
^{-2\alpha}$ on the subspace $\mathcal{H}_{j,k}$ is%
\[
E_{j,k}\left(  1-\frac{\left(  \alpha-1\right)  \left(  m+1-2\alpha\right)
\left(  2jk+n\left(  j+k-1+\alpha\right)  \right)  }{\left(  j-1+\alpha
\right)  \left(  j+m+1-\alpha\right)  \left(  k-1+\alpha\right)  \left(
k+m+1-\alpha\right)  }\right)
\]

\end{enumerate}

When $\alpha=0$ or $1$, the formulas are to be understood by taking limits
with fixed $j$ and $k$.
\end{proposition}

\begin{remark}
\label{pos}It is clear that $E_{j,k}>0$ for all $j$ and $k$ if $\alpha
=\lambda/4$ with $\lambda\in\left(  0,Q\right)  $. Therefore we draw the
following important corollary: the operator $I_{\lambda}$ is positive in the
sense that%
\[
\left\langle I_{\lambda}f,f\right\rangle =\int_{\mathbb{S}^{2m+1}%
\times\mathbb{S}^{2m+1}}\frac{\overline{f(\xi)}f\left(  \eta\right)
}{\left\vert 1-\xi\cdot\overline{\eta}\right\vert ^{\lambda/2}}d\sigma\left(
\xi\right)  d\sigma\left(  \eta\right)  \geq0.
\]

\end{remark}

From Proposition \ref{weights} Frank and Lieb deduced the following inequality
which plays a crucial role.

\begin{theorem}
\label{r2v}Let $0<\lambda<Q=2\left(  m+1\right)  $, then there exist $C>0$
s.t. for any $f$ on $\mathbb{S}^{2m+1}$ one has%
\begin{align*}
&  \int_{\mathbb{S}^{2m+1}}\frac{\overline{f\left(  \xi\right)  }f\left(
\eta\right)  \operatorname{Re}\xi\cdot\overline{\eta}}{\left\vert 1-\xi
\cdot\overline{\eta}\right\vert ^{\lambda/2}}d\sigma\left(  \eta\right)
d\sigma\left(  \xi\right) \\
&  \geq\frac{\lambda}{4\left(  m+1\right)  -\lambda}\left\langle I_{\lambda
}f,f\right\rangle +\frac{C\left(  2\left(  m+1\right)  -\lambda\right)
}{4\left(  m+1\right)  -\lambda}\left\langle I_{\lambda}\left(  f-a_{f}%
\right)  ,f-a_{f}\right\rangle ,
\end{align*}
where $a_{f}=\frac{1}{\left\vert \mathbb{S}^{2m+1}\right\vert }\int
_{\mathbb{S}^{2m+1}}f\left(  \eta\right)  d\sigma\left(  \eta\right)  $ is the
average of $f$.
\end{theorem}

\begin{remark}
Taking $C=0$, this is precisely Theorem 5.1 in \cite{FL1}. By inspecting their
proof, it is easy to get the above strengthened version.
\end{remark}

\begin{proposition}
\label{itr}If $f_{t,\eta}=f\circ\Phi_{t,\eta}\phi_{t,\eta}^{\left(
m+1\right)  /p}$, then%
\[
I_{\lambda}\left(  f_{t,\eta}\right)  =I_{\lambda}\left(  f\phi_{t,-\eta
}^{\left(  m+1\right)  /p^{\prime}-\lambda/4}\right)  \circ\Phi_{t,\eta}%
\phi_{t,\eta}^{\lambda/4}.
\]

\end{proposition}

\begin{proof}
By direction calculation, the following identity holds%
\begin{equation}
\left\vert 1-\Phi_{t,\eta}\left(  \xi\right)  \cdot\overline{\Phi_{t,\eta
}\left(  \zeta\right)  }\right\vert ^{2}=\left\vert 1-\xi\cdot\overline{\zeta
}\right\vert ^{2}\phi_{t,\eta}\left(  \xi\right)  \phi_{t,\eta}\left(
\zeta\right)  . \label{fid}%
\end{equation}

We compute by a change of variables%
\begin{align*}
I_{\lambda}\left(  f_{t,\eta}\right)  \left(  \xi\right)   &  =\int
_{\mathbb{S}^{2m+1}}\frac{f\circ\Phi_{t,\eta}\left(  \zeta\right)
\phi_{t,\eta}^{\left(  m+1\right)  /p}\left(  \zeta\right)  }{\left\vert
1-\xi\cdot\overline{\zeta}\right\vert ^{\lambda/2}}d\sigma\left(  \zeta\right)
\\
&  =\int_{\mathbb{S}^{2m+1}}\frac{f\left(  \zeta\right)  \left[  \phi_{t,\eta
}\circ\Phi_{t,\eta}^{-1}\left(  \zeta\right)  \right]  ^{-\left(  m+1\right)
/p^{\prime}}}{\left\vert 1-\xi\cdot\overline{\Phi_{t,\eta}^{-1}\left(
\zeta\right)  }\right\vert ^{\lambda/2}}d\sigma\left(  \zeta\right)  .
\end{align*}
It is easy to see that $\phi_{t,\eta}\circ\Phi_{t,\eta}^{-1}\left(
\zeta\right)  =1/\phi_{t,-\eta}\left(  \zeta\right)  $ while using (\ref{fid})
we have%
\begin{align*}
\left\vert 1-\xi\cdot\overline{\Phi_{t,\eta}^{-1}\left(  \zeta\right)
}\right\vert ^{2}  &  =\left\vert 1-\Phi_{t,\zeta}\left(  \xi\right)
\cdot\overline{\zeta}\right\vert ^{2}\phi_{t,-\eta}\circ\Phi_{t,\eta}\left(
\xi\right)  \left(  \xi\right)  \phi_{t,-\eta}\left(  \zeta\right) \\
&  =\left\vert 1-\Phi_{t,\zeta}\left(  \xi\right)  \cdot\overline{\zeta
}\right\vert ^{2}\phi_{t,-\eta}\left(  \zeta\right)  /\phi_{t,\eta}\left(
\xi\right)
\end{align*}
Therefore%
\begin{align*}
I_{\lambda}\left(  f_{t,\eta}\right)  \left(  \xi\right)   &  =\left[
\phi_{t,\eta}\left(  \xi\right)  \right]  ^{\lambda/4}\int_{\mathbb{S}^{2m+1}%
}\frac{f\left(  \zeta\right)  \phi_{t,-\eta}\left(  \zeta\right)  ^{\left(
m+1\right)  /p^{\prime}-\lambda/4}}{\left\vert 1-\Phi_{t,\eta}\left(
\xi\right)  \cdot\overline{\zeta}\right\vert ^{\lambda/2}\left(  \xi\right)
}d\sigma\left(  \zeta\right) \\
&  =I_{\lambda}\left(  f\phi_{t,-\eta}^{\left(  m+1\right)  /p^{\prime
}-\lambda/4}\right)  \circ\Phi_{t,\eta}\left(  \xi\right)  \phi_{t,\eta
}^{\lambda/4}\left(  \xi\right)  .
\end{align*}

\end{proof}

\section{\bigskip Proof of the sharp inequality}

We fix $\lambda\in\left(  0,Q\right)  $. Recall that the operator $I_{\lambda
}$ is defined by
\[
I_{\lambda}f\left(  \xi\right)  =\int_{\mathbb{S}^{2m+1}}\frac{f\left(
\eta\right)  }{\left\vert 1-\xi\cdot\overline{\eta}\right\vert ^{\lambda/2}%
}d\sigma\left(  \eta\right)  .
\]
Given $1<p<\frac{Q}{Q-\lambda}$, set $p^{\ast}=\frac{Qp}{Q-p\left(
Q-\lambda\right)  }>1$. By the work of Folland-Stein \cite{FS}, $I_{\lambda}$
is a bounded operator from $L^{p}\left(  \mathbb{S}^{2m+1}\right)  $ to
$L^{p^{\ast}}\left(  \mathbb{S}^{2m+1}\right)  $. In other words, there exists
a positive constant $C$ s.t. for all $f\in L^{p}\left(  \mathbb{S}%
^{2m+1}\right)  $.
\[
\left\Vert I_{\lambda}f\right\Vert _{p^{\ast}}\leq C\left\Vert f\right\Vert
_{p}.
\]
The contribution of Frank and Lieb \cite{FL1} (Theorem 2) is the determination
of the sharp constant and extremal functions when $p=\frac{2Q}{2Q-\lambda}$
and hence $p^{\ast}=\frac{2Q}{2Q+\lambda}$. Indeed, it is easy to verify that
the sharp inequality in Theorem 2 is equivalent to

\begin{theorem}
\label{flv}For $f\in L^{\frac{2Q}{2Q-\lambda}}\left(  \mathbb{S}%
^{2m+1}\right)  $%
\[
\left\Vert I_{\lambda}f\right\Vert _{\frac{2Q}{2Q+\lambda}}\leq\left(
\frac{2\pi^{m+1}}{m!}\right)  ^{\lambda/Q}\frac{m!\Gamma\left(  \left(
Q-\lambda\right)  /2\right)  }{\Gamma^{2}\left(  \left(  2Q-\lambda\right)
/4\right)  }\left\Vert f\right\Vert _{\frac{2Q}{2Q-\lambda}}.
\]

\end{theorem}

To present our proof of the above sharp inequality, we start with the following

\begin{proposition}
\label{compact}\bigskip For $1<p<\frac{Q}{Q-\lambda},1<q<p^{\ast}=\frac
{Qp}{Q-p\left(  Q-\lambda\right)  }$, the operator $I_{\lambda}:L^{p}\left(
\mathbb{S}^{2m+1}\right)  \rightarrow L^{q}\left(  \mathbb{S}^{2m+1}\right)  $
is compact.
\end{proposition}

This result is more or less standard. We outline the proof. Write $q=\theta
p+\left(  1-\theta\right)  p^{\ast}$ with $\theta\in\left(  0,1\right)  $. By
the Holder inequality, we have
\[
\left\Vert I_{\lambda}f\right\Vert _{q}\leq\left\Vert I_{\lambda}f\right\Vert
_{p}^{\theta}\left\Vert I_{\lambda}f\right\Vert _{p^{\ast}}^{1-\theta}%
\]
Therefore it suffices to prove that $I_{\lambda}$ from $L^{p}\left(
\mathbb{S}^{2m+1}\right)  $ to $L^{p}\left(  \mathbb{S}^{2m+1}\right)  $ is
compact. For any kernel $K\left(  \xi,\eta\right)  $, the following inequality
for the integral operator $I_{K}$ is well known (in a much more general
setting)%
\[
\left\Vert I_{K}f\right\Vert _{p}\leq C_{K}\left\Vert f\right\Vert _{p},
\]
where
\[
C_{K}=\max\left\{  \sup_{\xi}\int_{\mathbb{S}^{2m+1}}\left\vert K\left(
\xi,\eta\right)  \right\vert d\sigma\left(  \eta\right)  ,\sup_{\eta}%
\int_{\mathbb{S}^{2m+1}}\left\vert K\left(  \xi,\eta\right)  \right\vert
d\sigma\left(  \xi\right)  \right\}  .
\]
If $K$ is continuous, it can be approximated uniformly by polynomials in
$\xi,\eta$ by the Stone-Weierstrass theorem. Therefore $I_{K}:L^{p}\left(
\mathbb{S}^{2m+1}\right)  \rightarrow L^{p}\left(  \mathbb{S}^{2m+1}\right)  $
is compact as it can be approximated by operators of finite rank. In our case
$K\left(  \xi,\eta\right)  =\left\vert 1-\xi\cdot\overline{\eta}\right\vert
^{-\lambda/2}$. It is easy to see that it can be approximated by continuous
kernels
\[
K^{\varepsilon}\left(  \xi,\eta\right)  =\left\{
\begin{array}
[c]{cc}%
\varepsilon^{-\lambda/2}, & \text{if }\left\vert 1-\xi\cdot\overline{\eta
}\right\vert \leq\varepsilon,\\
\left\vert 1-\xi\cdot\overline{\eta}\right\vert ^{-\lambda/2} & \text{if
}\left\vert 1-\xi\cdot\overline{\eta}\right\vert \geq\varepsilon.
\end{array}
\right.
\]
Therefore $I_{\lambda}:L^{p}\left(  \mathbb{S}^{2m+1}\right)  \rightarrow
L^{p}\left(  \mathbb{S}^{2m+1}\right)  $ is compact.

\bigskip We now take $p>\frac{2Q}{2Q-\lambda}$. By a simple calculation, its
dual $p^{\prime}:=\frac{p}{p-1}<p^{\ast}$. Therefore the following
minimization problem
\begin{equation}
\Lambda_{p}=\sup\left\{  \left\Vert I_{\lambda}f\right\Vert _{p^{\prime}%
}/\left\Vert f\right\Vert _{p}:f\in L^{q}\left(  \mathbb{S}^{2m+1}\right)
,f\neq0\right\}  \label{mipcr}%
\end{equation}
has a solution $u_{p}$ by Proposition \ref{compact}. To simplify the
presentation, we will drop the subscript $p$ temporarily. We can obviously
assume that $u\geq0$. It satisfies the following Euler-Lagrange equation (when
properly scaled)%
\[
\left\{
\begin{array}
[c]{c}%
v^{p-1}=I_{\lambda}(u)\\
u^{p-1}=I_{\lambda}(v)
\end{array}
\right.  .
\]
Then%
\begin{align*}
\left\langle I_{\lambda}(u-v),u-v\right\rangle  &  =\left\langle
v^{p-1}-u^{p-1},u-v\right\rangle \\
&  =-\int_{\mathbb{S}^{2m+1}}\left(  u^{p-1}-v^{p-1}\right)  \left(
u-v\right)  d\sigma\\
&  \leq0.
\end{align*}
By the positivity of $I_{\lambda}$ (Remark \ref{pos}), we must have $u=v$,
i.e.
\[
I_{\lambda}(u)=u^{p-1}.
\]
Then
\[
\Lambda_{p}=\frac{\left\Vert I_{\lambda}u\right\Vert _{p^{\prime}}}{\left\Vert
u\right\Vert _{p}}=\left\Vert u\right\Vert _{p}^{p-2}%
\]
or $\left\Vert u\right\Vert _{p}=\Lambda_{p}^{1/\left(  p-2\right)  }$. It is
clear that $\lim_{p\rightarrow\frac{2Q}{2Q-\lambda}}\Lambda_{p}^{1/\left(
p-2\right)  }=\Lambda^{-2\left(  Q-\lambda\right)  /\left(  2Q-\lambda\right)
}$.

\begin{lemma}
\label{m0}The function $u$ satisfies%
\[
\int_{\mathbb{S}^{2m+1}}u\left(  \xi\right)  ^{p}\xi d\sigma\left(
\xi\right)  =0.
\]

\end{lemma}

\begin{proof}
We consider for any $\eta\in\mathbb{S}^{2m+1}$ the family $u_{t,\eta}%
=u\circ\Phi_{t,\eta}\phi_{t,\eta}^{\left(  m+1\right)  /p}$. Clearly
$\left\Vert u_{t,\eta}\right\Vert _{p}=\left\Vert u\right\Vert _{p}$. By
Proposition \ref{itr}, we have
\[
I_{\lambda}\left(  u_{t,\eta}\right)  =I_{\lambda}\left(  u\phi_{t,-\eta
}^{\left(  m+1\right)  /p^{\prime}-\lambda/4}\right)  \circ\Phi_{t,\eta}%
\phi_{t,\eta}^{\lambda/4}.
\]
Therefore
\begin{align*}
\left\Vert I_{\lambda}\left(  u_{t}\right)  \right\Vert _{p^{\prime}%
}^{p^{\prime}}  &  =\int_{\mathbb{S}^{2m+1}}\phi_{t,\eta}^{\lambda p^{\prime
}/4}\left(  \xi\right)  \left(  I_{\lambda}\left(  u\phi_{t,-\eta}^{\left(
m+1\right)  /p^{\prime}-\lambda/4}\right)  \circ\Phi_{t,\eta}\left(
\xi\right)  \right)  ^{p^{\prime}}d\sigma\left(  \xi\right) \\
&  =\int_{\mathbb{S}^{2m+1}}\phi_{t,\eta}^{\lambda p^{\prime}/4-m+1}\circ
\Phi_{t,\eta}^{-1}\left(  \xi\right)  \left(  I_{\lambda}\left(
u\phi_{t,-\eta}^{\left(  m+1\right)  /p^{\prime}-\lambda/4}\right)  \left(
\xi\right)  \right)  ^{p^{\prime}}d\sigma\left(  \xi\right) \\
&  =\int_{\mathbb{S}^{2m+1}}\phi_{t,-\eta}^{m+1-\lambda p^{\prime}/4}\left(
\xi\right)  \left(  I_{\lambda}\left(  u\phi_{t,-\eta}^{\left(  m+1\right)
/p^{\prime}-\lambda/4}\right)  \left(  \xi\right)  \right)  ^{p^{\prime}%
}d\sigma\left(  \xi\right)  .
\end{align*}
Differentiating $\log\left(  \left\Vert I_{\lambda}\left(  u_{t}\right)
\right\Vert _{p^{\prime}}/\left\Vert u_{t}\right\Vert _{p}\right)  $ at $t=0$
yields,
\begin{align*}
0  &  =\int_{\mathbb{S}^{2m+1}}\left[  \overset{\cdot}{\left(  m+1-\frac
{\lambda p^{\prime}}{4}\right)  \phi}I_{\lambda}(u)^{p^{\prime}}+p^{\prime
}\left(  \frac{m+1}{p^{\prime}}-\frac{\lambda}{4}\right)  I_{\lambda
}(u)^{p^{\prime}-1}I_{\lambda}(u\overset{\cdot}{\phi})\right]  d\sigma\\
&  =\left(  m+1-\frac{\lambda p^{\prime}}{4}\right)  \int_{\mathbb{S}^{2m+1}%
}\left[  I_{\lambda}(u)^{p^{\prime}}\overset{\cdot}{\phi}+I_{\lambda
}(u)^{p^{\prime}-1}I_{\lambda}(u\overset{\cdot}{\phi})\right]  d\sigma,
\end{align*}
where
\[
\overset{\cdot}{\phi}\left(  \xi\right)  =\frac{d}{dt}|_{t=0}\phi_{t,-\eta
}\left(  \xi\right)  =2\operatorname{Re}\left(  \xi\cdot\overline{\eta
}\right)  .
\]
By the Euler-Lagrange equation, $I_{\lambda}(u)=u^{p-1}$. Therefore we have%
\begin{align*}
0  &  =\int_{\mathbb{S}^{2m+1}}\left[  u^{p}\overset{\cdot}{\phi}+uI_{\lambda
}(u\overset{\cdot}{\phi})\right]  d\sigma\\
&  =\int_{\mathbb{S}^{2m+1}}\left[  u^{p}\overset{\cdot}{\phi}+I_{\lambda
}\left(  u\right)  (u\overset{\cdot}{\phi})\right]  d\sigma\\
&  =2\int_{\mathbb{S}^{2m+1}}u^{p}\overset{\cdot}{\phi}d\sigma,
\end{align*}
i.e. $\int_{\mathbb{S}^{2m+1}}u^{p}\left(  \xi\right)  \operatorname{Re}%
\left(  \xi\cdot\overline{\eta}\right)  d\sigma\left(  \xi\right)  =0$ for all
$\eta\in\mathbb{S}^{2m+1}$. The conclusion follows.
\end{proof}

\bigskip By the 2nd variation, we have for all real $f$ with $\int
_{\mathbb{S}^{n}}u^{p-1}f=0$%
\[
\left(  p^{\prime}-1\right)  \int_{\mathbb{S}^{2m+1}}\left(  I_{\lambda
}u\right)  ^{p^{\prime}-2}\left(  I_{\alpha}f\right)  ^{2}\leq\left(
p-1\right)  ^{2}\int_{\mathbb{S}^{2m+1}}u^{p-2}f^{2}.
\]
Using the Euler-Lagrange equation $I_{\lambda}(u)=u^{p-1}$, this simplifies as%
\begin{equation}
\int_{\mathbb{S}^{2m+1}}u^{2-p}\left(  I_{\lambda}f\right)  ^{2}\leq\left(
p-1\right)  ^{2}\int_{\mathbb{S}^{2m+1}}u^{p-2}f^{2}.\label{2ndv}%
\end{equation}
By the Holder inequality%
\begin{align*}
\int_{\mathbb{S}^{2m+1}}fI_{\lambda}f &  =\int_{\mathbb{S}^{2m+1}}u^{\left(
p-2\right)  /2}fu^{\left(  2-p\right)  /2}I_{\lambda}f\\
&  \leq\left(  \int_{\mathbb{S}^{2m+1}}u^{p-2}f^{2}\right)  ^{1/2}\left(
\int_{\mathbb{S}^{2m+1}}u^{2-p}\left(  I_{\lambda}f\right)  ^{2}\right)
^{1/2}\\
&  \leq\left(  p-1\right)  \int_{\mathbb{S}^{2m+1}}u^{p-2}f^{2}.
\end{align*}
In summary, we have for all real $f$ with $\int_{\mathbb{S}^{n}}u^{p-1}f=0$%
\[
\int_{\mathbb{S}^{2m+1}\times\mathbb{S}^{2m+1}}\frac{f(\xi)f\left(
\eta\right)  }{\left\vert 1-\xi\cdot\overline{\eta}\right\vert ^{\lambda/2}%
}d\sigma\left(  \xi\right)  d\sigma\left(  \eta\right)  \leq\left(
p-1\right)  \int_{\mathbb{S}^{2m+1}}u^{p-2}f^{2}.
\]
By Lemma \ref{m0}, we can take $f\left(  z\right)  =u\left(  z\right)
\operatorname{Re}z_{i}$ or $u\left(  z\right)  \operatorname{Im}%
z_{i},i=1,\cdots,m+1$ and adding these inequalities yields%
\[
\int_{\mathbb{S}^{2m+1}\times\mathbb{S}^{2m+1}}\frac{\operatorname{Re}\xi
\cdot\overline{\eta}}{|\xi-\eta|^{\lambda/2}}u(\xi)u(\eta)d\sigma\left(
\xi\right)  d\sigma\left(  \eta\right)  \leq\left(  p-1\right)  \int
_{\mathbb{S}^{2m+1}}u^{p}.
\]
Combined with Theorem \ref{r2v}, this implies%
\begin{align*}
\left(  p-1\right)  \int_{\mathbb{S}^{2m+1}}u^{p} &  \geq\frac{\lambda
}{4\left(  m+1\right)  -\lambda}\left\langle I_{\lambda}u,u\right\rangle
+\frac{C\left(  2\left(  m+1\right)  -\lambda\right)  }{4\left(  m+1\right)
-\lambda}\left\langle I_{\lambda}\left(  u-a\right)  ,\left(  u-a\right)
\right\rangle \\
&  =\frac{\lambda}{4\left(  m+1\right)  -\lambda}\int_{\mathbb{S}^{2m+1}}%
u^{p}+\frac{C\left(  2\left(  m+1\right)  -\lambda\right)  }{4\left(
m+1\right)  -\lambda}\left\langle I_{\lambda}\left(  u-a\right)  ,\left(
u-a\right)  \right\rangle ,
\end{align*}
where $a$ is the average of $u$. Therefore we have%
\[
\left(  p-\frac{4\left(  m+1\right)  }{4\left(  m+1\right)  -\lambda}\right)
\int_{\mathbb{S}^{2m+1}}u^{p}\geq\frac{C\left(  2\left(  m+1\right)
-\lambda\right)  }{4\left(  m+1\right)  -\lambda}\left\langle I_{\lambda
}\left(  u-a\right)  ,\left(  u-a\right)  \right\rangle .
\]
Note that $\frac{4\left(  m+1\right)  }{4\left(  m+1\right)  -\lambda}%
=\frac{2Q}{2Q-\lambda}$.

By the positivity of $I_{\lambda}$, the RHS\ is nonnegative. As a consequence
(now reattaching the subscript $p$), we have $\left\langle I_{\lambda}\left(
u_{p}-a_{p}\right)  ,\left(  u_{p}-a_{p}\right)  \right\rangle \rightarrow0$
as $p\searrow\frac{2Q}{2Q-\lambda}$. By the Euler-Lagrange equation
$I_{\lambda}(u)=u^{p-1}$again and the observation that $I_{\lambda}$ maps a
constant function to a constant function, this means $\int_{\mathbb{S}^{2m+1}%
}u_{p}^{p-1}\left(  u_{p}-a_{p}\right)  \rightarrow0$ as $p\searrow\frac
{2Q}{2Q-\lambda}$. Thus, as $p\searrow\frac{2Q}{2Q-\lambda}$
\begin{align*}
\int_{\mathbb{S}^{2m+1}}u_{p}^{p} &  =a_{p}\int_{\mathbb{S}^{2m+1}}u_{p}%
^{p-1}+o\left(  1\right)  \\
&  \leq a_{p}\left\vert \mathbb{S}^{2m+1}\right\vert ^{1/p}\left(
\int_{\mathbb{S}^{2m+1}}u_{p}^{p}\right)  ^{\left(  p-1\right)  /p}+o\left(
1\right)  .
\end{align*}
This implies%
\[
\left\Vert u_{p}\right\Vert _{p}\leq a_{p}\left\vert \mathbb{S}^{2m+1}%
\right\vert ^{1/p}+o\left(  1\right)  .
\]
On the other hand, we have $\left\Vert u_{p}\right\Vert _{p}\geq
a_{p}\left\vert \mathbb{S}^{2m+1}\right\vert ^{1/p}$ by the Holder inequality.
Therefore $\lim_{p\searrow\frac{2Q}{2Q-\lambda}}\left\Vert u_{p}\right\Vert
_{p}-a_{p}\left\vert \mathbb{S}^{2m+1}\right\vert ^{1/p}=0$. We assume that
$u_{p}\overset{w^{\ast}}{\longrightarrow}v$ in $L^{\frac{2Q}{2Q-\lambda}%
}(\mathbb{S}^{2m+1})$. Clearly, $\lim_{p\searrow\frac{2Q}{2Q-\lambda}}a_{p}%
=a$, the average of $v$. We have
\begin{align*}
\left\Vert v\right\Vert _{\frac{2Q}{2Q-\lambda}} &  \leq\lim_{p\searrow
\frac{2Q}{2Q-\lambda}}\left\Vert u_{p}\right\Vert _{\frac{2Q}{2Q-\lambda}}\\
&  \leq\lim_{p\searrow\frac{2Q}{2Q-\lambda}}\left\Vert u_{p}\right\Vert
_{p}\left\vert \mathbb{S}^{2m+1}\right\vert ^{\left(  p-\frac{2Q}{2Q-\lambda
}\right)  /p}\\
&  =\lim_{p\searrow\frac{2Q}{2Q-\lambda}}a_{p}\left\vert \mathbb{S}%
^{2m+1}\right\vert ^{1/p}\\
&  =a\left\vert \mathbb{S}^{2m+1}\right\vert ^{\frac{2Q-\lambda}{2Q}}.
\end{align*}
In view of the Holder inequality, we must have $v=a$, i.e. $v$ is constant and
all the inequalities in the formula above are equalities, i.e.
\[
a\left\vert \mathbb{S}^{2m+1}\right\vert ^{\frac{2Q-\lambda}{2Q}}=\left\Vert
v\right\Vert _{\frac{2Q}{2Q-\lambda}}=\lim_{p\searrow\frac{2Q}{2Q-\lambda}%
}\left\Vert u_{p}\right\Vert _{\frac{2Q}{2Q-\lambda}}.
\]
The weak $\ast$-convergence plus the convergence of the norms implies strong
convergence $u_{p}\rightarrow a$ in $L^{\frac{2Q}{2Q-\lambda}}(\mathbb{S}%
^{2m+1})$. Therefore the constant function is a minimizer for
\[
\inf\left\Vert I_{\lambda}f\right\Vert _{\frac{2Q}{2Q+\lambda}}/\left\Vert
f\right\Vert _{\frac{2Q}{2Q-\lambda}}.
\]
Theorem \ref{flv} then follows from a simple calculation.

\section{Further remarks}

In a later paper \cite{FL2}, Frank and Lieb showed that the new method
developed in \cite{FL1} can be adapted to give a new, rearrangement-free proof
of the following sharp Hardy-Littlewood-Sobolev inequality on $\mathbb{R}^{n}$
which was proved originally by Lieb \cite{L} using rearrangement arguments.

\begin{theorem}
Let $0<\lambda<n$ and $p=2n/\left(  2n-\lambda\right)  $. Then for any $f,g\in
L^{p}\left(  \mathbb{S}^{n}\right)  $%
\[
\left\vert \int_{\mathbb{S}^{n}\times\mathbb{S}^{n}}\frac{\overline{f(\xi
)}g\left(  \eta\right)  }{\left\vert \xi-\eta\right\vert ^{\lambda}}%
d\sigma\left(  \xi\right)  d\sigma\left(  \eta\right)  \right\vert \leq
\pi^{\lambda/2}\frac{\Gamma\left(  \left(  n-\lambda\right)  /2\right)
}{\Gamma\left(  n-\lambda/2\right)  }\left(  \frac{\Gamma\left(  n\right)
}{\Gamma\left(  n/2\right)  }\right)  ^{1-\lambda/n}\left\Vert f\right\Vert
_{p}\left\Vert g\right\Vert _{p},
\]
with equality if and only if%
\[
f\left(  \xi\right)  =\frac{c}{\left\vert 1-\xi\cdot a\right\vert ^{\left(
2n-\lambda\right)  /2}},g\left(  \xi\right)  =\frac{c^{\prime}}{\left\vert
1-\xi\cdot a\right\vert ^{\left(  2n-\lambda\right)  /2}}%
\]
for some $c,c^{\prime}\in\mathbb{C}$ and $a\in\mathbb{R}^{n+1}$ with
$\left\vert a\right\vert <1$ (unless $f\equiv0$ or $g\equiv0$).
\end{theorem}

Our method can also be adapted to give a simpler proof of Lieb's theorem. In
this case, we work with the operator
\[
I_{\lambda}f\left(  \xi\right)  =\int_{\mathbb{S}^{n}}\frac{f\left(
\eta\right)  }{\left\vert \xi-\eta\right\vert ^{\lambda}}d\sigma\left(
\eta\right)
\]
and consider, for $p>2n/\left(  2n-\lambda\right)  $, the minimization
problem
\begin{equation}
\Lambda_{p}=\sup\left\{  \left\Vert I_{\lambda}f\right\Vert _{p^{\prime}%
}/\left\Vert f\right\Vert _{p}:f\in L^{p}\left(  \mathbb{S}^{n}\right)
,f\neq0\right\}  . \label{mir}%
\end{equation}
The rest of proof requires minor modifications and we omit the details.

We end with some open problems. In the CR cases, it would be interesting to
classify all positive solutions to the Euler-Lagrange equation%
\[
I_{\lambda}(u)=u^{p-1}%
\]
for all $p\geq2Q/\left(  2Q-\lambda\right)  $ on $\mathbb{S}^{2m+1}$, not
merely extremal functions of the corresponding inequality. On $\mathbb{S}^{n}%
$, this kind of classification results can be established by the powerful
method of moving planes or moving spheres (cf. \cite{CLO, Li}). On
$\mathbb{S}^{2m+1}$, the classification was known in the critical case
$p=2Q/\left(  2Q-\lambda\right)  $ only when $\lambda=Q-2$ by the work of
Jerison-Lee \cite{JL2}. The critical case when $\lambda\neq Q-2$ and all the
subcritical cases seem to be largely open on $\mathbb{S}^{2m+1}$ as far as we
know (cf. \cite{W1,W2} for discussions about the significance of such
classification problems).

\end{document}